\documentclass[12pt]{article}
\usepackage[utf8]{inputenc}
\usepackage{amsmath}
\usepackage{amsthm}
\usepackage{amssymb}
\usepackage{geometry}
\usepackage{natbib}
\makeatletter
\usepackage{amsthm}
\usepackage{url}
\newcommand{\norm}[1]{\left\| #1 \right\|}
\newcommand{\abs}[1]{\left| #1 \right|}
\newcommand{\E}{\mathbb{E}}
\DeclareMathOperator{\Var}{Var}

\newcommand{\ind}[1]{\mathbf{1}_{\{#1\}}}
\makeatother

\theoremstyle{plain}
\newtheorem{thm}{\protect\theoremname}
\newtheorem{lem}{\protect\lemmaname}
\providecommand{\lemmaname}{Lemma}
\providecommand{\theoremname}{Theorem}

\begin{document}
\title{A Sharp Small-Coefficient  Variant of Khintchine's Inequality and the Sharp $\pi/2$
Theorem}
\author{Lei Yu\thanks{L. Yu is with the School of Statistics and Data Science, LPMC, KLMDASR,
and LEBPS, Nankai University, Tianjin 300071, China (e-mail: leiyu@nankai.edu.cn).
This work was supported by the National Key Research and Development
Program of China under grant 2023YFA1009604 and the NSFC under grant
62101286.}}
\maketitle
\begin{abstract}
We prove a refined quadratic normal approximation for the first absolute
moment of normalized weighted Rademacher sums with bounded maximal
coefficients. For any weight vector $w\in\mathbb{R}^{n}$ satisfying
$\norm{w}_{2}=1$ and $\norm{w}_{\infty}\leq\beta$ with sufficiently
small $\beta>0$, we establish the uniform error bound $\left|\E\abs{\sum^{n}_{i=1}w_{i}X_{i}}-\sqrt{2/\pi}\right|=O(\beta^{2})$
over all admissible weight configurations. Our proof combines zero-bias
Stein’s method and refined small-ball probability estimates to exploit
symmetry cancellation and control the non-smooth residual of the absolute-value
test function. An explicit extremal construction further verifies
the optimality of this quadratic convergence rate. As an application,
we establish an asymptotically sharp refinement of the Friedgut--Kalai--Naor
(FKN) theorem for Boolean functions, also known as the sharp $\pi/2$
theorem, characterizing the level-1 Fourier energy for functions deviating
far from dictatorships. 
\end{abstract}

\section{Introduction}

Let $X_{1},X_{2},\dots,X_{n}$ be independent identically distributed
Rademacher random variables, i.e., $\mathbb{P}(X_{i}=1)=\mathbb{P}(X_{i}=-1)=1/2$.
Given a weight vector $w=(w_{1},\dots,w_{n})\in\mathbb{R}^{n}$, define
the normalized weighted Rademacher sum 
\[
W=\sum^{n}_{i=1}w_{i}X_{i},
\]
subject to the normalization condition $\norm{w}^{2}_{2}=\sum^{n}_{i=1}w^{2}_{i}=1$
and maximal coefficient bound $\norm{w}_{\infty}=\max_{1\leq i\leq n}\abs{w_{i}}\leq\beta$.
By the classical central limit theorem, $W$ converges in distribution
to a standard Gaussian random variable $Z\sim\mathcal{N}(0,1)$. A
fundamental Gaussian approximation for symmetric distributions is
the first absolute moment matching: $\E\abs{W}\to\E\abs{Z}=\sqrt{2/\pi}$
as $\beta\to0$. König, Schütt, and Tomczak-Jaegermann \cite{konig1999projection}
proved the following variant of Khintchine's inequality 
\begin{equation}
\left|\mathbb{E}|W|-\sqrt{2/\pi}\right|\le\Big(1-\sqrt{\frac{2}{\pi}}\Big)\,\beta,\label{eq:-3}
\end{equation}
which can be also recovered with an implicit coefficient $C$ by applying
the Berry--Esseen Theorem with proper truncation argument \cite[Theorem 5.16]{ODonnell14analysisof}.

In this work, we rigorously prove that the absolute moment approximation
error decays quadratically in $\beta$, establishing a sharp improvement
over the bound above. We define 
\[
G(\beta)=\sup\left\{ \E\abs{\sum^{n}_{i=1}w_{i}X_{i}}:n\geq1,\ \norm{w}_{2}=1,\ \norm{w}_{\infty}\leq\beta\right\} ,
\]
which quantifies the worst-case Gaussian approximation error over
all admissible weighted Rademacher sums. Our main theorem establishes
the optimal quadratic expansion of $G(\beta)$, uniform over all valid
weight configurations. 
\begin{thm}[Sharp Small-Coefficient  Variant of Khintchine's Inequality]
\label{thm:main} There exist universal constants $C,\beta_{0}>0$
such that for all $0<\beta\leq\beta_{0}$, all integers $n\geq1$,
and all weight vectors $w\in\mathbb{R}^{n}$ with $\norm{w}_{2}=1$
and $\norm{w}_{\infty}\leq\beta$, 
\begin{equation}
\left|\E\abs{W}-\sqrt{\frac{2}{\pi}}\right|\leq C\beta^{2}.\label{eq:-13}
\end{equation}
Consequently, $G(\beta)=\sqrt{2/\pi}+O(\beta^{2})$, and the quadratic
rate is optimal in the order. 
\end{thm}

To confirm the quadratic rate is sharp, construct the uniform weight
vector $w_{i}=1/\sqrt{n}$ with $n=\lceil\beta^{-2}\rceil$. This
satisfies $\norm{w}_{2}=1$ and $\norm{w}_{\infty}=1/\sqrt{n}\leq\beta$.
Applying Stirling's formula yields 
\[
\E\abs{\frac{1}{\sqrt{n}}\sum^{n}_{i=1}X_{i}}\geq\sqrt{\frac{2}{\pi}}-C\beta^{2}+O(\beta^{4}),
\]
with $C=\frac{1}{2\sqrt{\pi}}$ for even $n$ and $C=\frac{3}{2\sqrt{2\pi}}$ for odd $n$,
providing a matching upper bound. Thus $G(\beta)=\sqrt{2/\pi}+O(\beta^{2})$,
and the $O(\beta^{2})$ rate is optimal.

We next prove \eqref{eq:-13} by combining Stein’s method and the
classical CLT. 

\section{Preliminaries}

We recap core tools from Stein’s method and zero-bias transformation
required for the main proof, including regularity properties of the
Stein equation solution for the absolute value function. 

\subsection{Zero-Bias Transformation}

Stein’s characterization for standard normality states that a random
variable $Z\sim\mathcal{N}(0,1)$ if and only if 
\[
\E[Zf(Z)]=\E[f'(Z)]
\]
for all absolutely continuous functions $f$ with integrable derivatives.
For any mean-zero, unit-variance random variable $W$, its zero-bias
transform $W^{*}$ is the unique random variable satisfying 
\[
\E[Wf(W)]=\E[f'(W^{*})]
\]
for all admissible $f$. For weighted Rademacher sums $W=\sum^{n}_{i=1}w_{i}X_{i}$
with $\norm{w}_{2}=1$, we construct $W^{*}$ explicitly. Let $U_{1},\dots,U_{n}$
be independent uniform random variables on $[-1,1]$, independent
of $\{X_{i}\}$. The zero-bias transform of each summand $w_{i}X_{i}$
is $w_{i}U_{i}$, verified via integration by parts. Let $I$ be an
independent random index with $\mathbb{P}(I=i)=w^{2}_{i}$, and define
leave-one-out sums 
\[
Y_{i}=\sum_{j\neq i}w_{j}X_{j}.
\]
The zero-bias transform of $W$ is then 
\[
W^{*}=Y_{I}+w_{I}U_{I}.
\]
This explicit coupling is a standard tool for refined normal approximation
of symmetric sums \cite{dobler2015stein}. 

\subsection{Stein Equation for the Absolute Value Function}

Let $a=\E\abs{Z}=\sqrt{2/\pi}$ denote the standard Gaussian first
absolute moment. We solve the Gaussian Stein equation for the test
function $h(x)=|x|$: 
\begin{equation}
f'(x)-xf(x)=|x|-a.\label{eq:stein}
\end{equation}
The unique bounded solution to \eqref{eq:stein} is 
\begin{equation}
f(x)=e^{x^{2}/2}\int^{x}_{-\infty}\left(|t|-a\right)e^{-t^{2}/2}dt.\label{eq:stein-sol}
\end{equation}
The solution possesses uniform derivative regularity away from the
origin, critical for our local error analysis.
\begin{lem}
\label{lem:stein-reg} There exists a universal constant $C_{0}>0$
such that the Stein solution $f$ in \eqref{eq:stein-sol} satisfies
\[
\sup_{x\neq0}\abs{f''(x)}\leq C_{0},\qquad\sup_{x\neq0}\abs{f'''(x)}\leq C_{0}.
\] 
\end{lem}

\begin{proof}
For $x>0$, using $\int^{\infty}_{-\infty}(|t|-a)e^{-t^{2}/2}\,dt=0$,
we may rewrite \eqref{eq:stein-sol} as 
\begin{equation}
f(x)=-e^{x^{2}/2}\int^{\infty}_{x}(t-a)e^{-t^{2}/2}\,dt=-1+aR(x),\label{eq:}
\end{equation}
where 
\[
R(x):=e^{x^{2}/2}\int^{\infty}_{x}e^{-t^{2}/2}\,dt.
\]
Substituting \eqref{eq:} into \eqref{eq:stein} yields 
\begin{equation}
f'(x)=x\left(-1+aR(x)\right)+x-a=axR(x)-a.\label{eq:-1}
\end{equation}
By the Mills ratio inequality, 
$
0\le xR(x)\le1
$ for $x>0$,
so $|f'(x)|\le a$ for $x>0$. The case $x<0$ follows by symmetry
since $f$ is odd, hence $f'$ is even. Thus $\sup_{x\neq0}\abs{f''(x)}\le a<\infty$.

We now prove the bound on $f''$. Differentiating \eqref{eq:stein}
for $x\neq0$ gives 
\begin{equation}
f''(x)-f(x)-xf'(x)=\operatorname{sgn}(x).\label{eq:-2}
\end{equation}
Using \eqref{eq:} and \eqref{eq:-1}, we obtain for $x>0$,
\[
f''(x)=a(1+x^{2})R(x)-ax.
\]
Using the standard Mills ratio expansion $R(x)=x^{-1}-x^{-3}+O(x^{-5})$,
we obtain 
\[
f''(x)=a(1+x^{2})\left(x^{-1}-x^{-3}+O(x^{-5})\right)-ax=-ax^{-3}+O(x^{-3})=O(x^{-3}).
\]
Thus $f''$ is bounded on $(0,\infty)$, and by symmetry on $(-\infty,0)$.
Hence $\|f''\|_{\infty}\le C_{0}$.

Finally, we establish the bound on $f'''$ away from $0$. Differentiating
\eqref{eq:-2} for $x\neq0$ gives 
\[
f'''(x)=xf''(x)+2f'(x).
\]
Since $f''(x)=O(|x|^{-3})$, we have 
\[
|f'''(x)|\le|x|\,|f''(x)|+2|f'(x)|\le C'+2\norm{f'}_{\infty}<\infty,
\]
uniformly for $x\neq0$. Hence $\sup_{x\neq0}|f'''(x)|\le C_{1}$.
\end{proof}

\subsection{Taylor Estimate}

The only non-smoothness of $f'$ occurs at $x=0$, which dominates
the residual approximation error. To quantify this singularity, define
the truncated second derivative 
\[
d(y)=\begin{cases}
f''(y) & y\neq0,\\
0 & y=0.
\end{cases}
\]

\begin{lem}
\label{lem:local-taylor} There exists a universal constant $C_{1}>0$
such that for all $y,u\in\mathbb{R}$, 
\begin{equation}
\abs{f'(y+u)-f'(y)-u\,d(y)}\leq C_{1}\left(u^{2}+\abs{u}\ind{\abs{y}\leq\abs{u}}\right).\label{eq:local-est}
\end{equation}
\end{lem}

\begin{proof}
Case 1: The line segment connecting $y$ and $y+u$ does not contain
$0$. Then $f'''$ is uniformly bounded on the segment by Lemma \ref{lem:stein-reg},
and standard Taylor expansion yields a remainder of order $O(u^{2})$.

Case 2: The segment crosses $0$. This forces $\abs{y}\leq\abs{u}$.
Since $f'$ is absolutely continuous and $f''$ are bounded a.e., as a standard consequence of the fundamental theorem of calculus for absolutely continuous functions,
$f'$ is Lipschitz continuous with Lipschitz constant at most $K=\sup_{x\neq0}\abs{f''(x)}<\infty$.
We obtain the bound $\abs{f'(y+u)-f'(y)}+\abs{ud(y)}\leq2K\abs{u}$,
which is absorbed by the indicator term $\abs{u}\ind{\abs{y}\leq\abs{u}}$.

Combining both cases establishes \eqref{eq:local-est}. 
\end{proof}

\section{Proof of Theorem \ref{thm:main}}

We decompose the approximation error via the zero-bias Stein identity,
bound incremental differences between Rademacher and uniform variables,
and control small-ball probabilities via the Berry--Esseen theorem
to yield quadratic error decay. 

For fixed index $i$ and state $y\in\mathbb{R}$, define the increment
discrepancy 
\[
A_{i}(y)=\E_{X}\left[f'(y+w_{i}X)\right]-\E_{U}\left[f'(y+w_{i}U)\right],
\]
where $X$ is Rademacher and $U$ is uniform on $[-1,1]$. Both variables
have zero mean, so linear Taylor terms vanish under expectation. Applying
Lemma \ref{lem:local-taylor} with $y+u=w_{i}X$ and $y=w_{i}U$ yields
\begin{equation}
\abs{A_{i}(y)}\leq C_{2}w^{2}_{i}+C_{2}\abs{w_{i}}\ind{\abs{y}\leq\abs{w_{i}}}.\label{eq:-4}
\end{equation}
By the Stein identity and zero-bias characterization, 
\[
\E\abs{W}-a=\E\left[f'(W)-Wf(W)\right]=\E f'(W)-\E f'(W^{*}).
\]
Substituting the zero-bias construction of $W^{*}$, we expand the
error: 
\begin{equation}
\begin{aligned}\abs{\E\abs{W}-a} & \leq\sum^{n}_{i=1}w^{2}_{i}\E\left|A_{i}(Y_{i})\right|\\
 & \leq C_{2}\sum^{n}_{i=1}w^{4}_{i}+C_{2}\sum^{n}_{i=1}\abs{w_{i}}^{3}\mathbb{P}\left(\abs{Y_{i}}\leq\abs{w_{i}}\right).
\end{aligned}
\label{eq:-5}
\end{equation}

We bound the small-ball probabilities $\mathbb{P}(\abs{Y_{i}}\leq\abs{w_{i}})$
using the classical Berry--Esseen theorem. The variance of leave-one-out
sums is 
\[
\sigma^{2}_{i}=\Var(Y_{i})=1-w^{2}_{i}.
\]
For $\beta\leq1/2$, we have $\sigma^{2}_{i}\geq3/4$. The third absolute
moment of the summands of $Y_{i}$ satisfies 
\[
\rho_{i}=\sum_{j\neq i}\abs{w_{j}}^{3}\leq\norm{w}_{\infty}\sum_{j\neq i}w^{2}_{j}\leq\beta.
\]
The Berry--Esseen inequality gives a uniform distributional error
bound: 
\begin{equation}
\sup_{x\in\mathbb{R}}\abs{\mathbb{P}(Y_{i}\leq x)-\Phi(x/\sigma_{i})}\leq C_{\text{BE}}\frac{\rho_{i}}{\sigma^{3}_{i}}\leq C_{3}\beta.\label{eq:-6}
\end{equation}

Since $\sigma_{i}\ge\sqrt{3/4}$ and $|w_{i}|\le\beta$, there exists
a universal constant $C_{4}>0$ such that $\mathbb{P}\left(|Z|\le|w_{i}|/\sigma_{i}\right)\le C_{4}|w_{i}|\le C_{4}\beta$.
Combining with the Berry--Esseen error term, we obtain 
\begin{equation}
\begin{aligned}\mathbb{P}\left(\abs{Y_{i}}\leq\abs{w_{i}}\right) & \leq\mathbb{P}\left(\abs{Z}\leq\frac{\abs{w_{i}}}{\sigma_{i}}\right)+2C_{3}\beta\\
 & \leq C_{4}\beta+2C_{3}\beta\,\leq C_{5}\beta.
\end{aligned}
\label{eq:-7}
\end{equation}

Using the normalization and maximal weight bound: 
\begin{equation}
\sum^{n}_{i=1}w^{4}_{i}\leq\norm{w}^{2}_{\infty}\sum^{n}_{i=1}w^{2}_{i}\leq\beta^{2},\qquad\sum^{n}_{i=1}\abs{w_{i}}^{3}\leq\norm{w}_{\infty}\sum^{n}_{i=1}w^{2}_{i}\leq\beta.\label{eq:-8}
\end{equation}
Substituting \eqref{eq:-7} and \eqref{eq:-8} into \eqref{eq:-5}
yields 
\[
\begin{aligned}\abs{\E\abs{W}-a} & \leq C_{2}\beta^{2}+C_{2}C_{5}\beta\cdot\beta\leq C\beta^{2}.\end{aligned}
\]
This establishes the upper bound of Theorem \ref{thm:main}. 

\section{Application to Strengthening FKN Theorem}

As a direct and significant application of our quadratic Gaussian
approximation theorem for weighted Rademacher sums, we derive a refined
level-1 Fourier weight bound for Boolean functions, which strengthens
the classical Friedgut--Kalai--Naor (FKN) theorem in the regime
where the maximal absolute value of single-coordinate Fourier coefficient tends to zero,
namely, functions deviate far from dictatorships.  

\subsection{The Sharp $\frac{\pi}{2}$ Theorem}

Let $f:\{\pm1\}^{n}\to\{\pm1\}$ be a Boolean function. Denote its
level-1 Fourier weight and maximal level-1 Fourier coefficient by
\[
W_{1}(f):=\sum^{n}_{i=1}\widehat{f}^{2}_{\{i\}},\qquad\theta:=\max_{1\le i\le n}\big|\widehat{f}_{\{i\}}\big|,
\]
where $\widehat{f}_{S}=\mathbb{E}_{\mathbf{X}}[f(\mathbf{X})\chi_{S}(\mathbf{X})]$
are the standard Boolean Fourier coefficients and $\chi_{S}(\mathbf{X})=\prod_{i\in S}X_{i}$
are Rademacher characters.  

We state the refined FKN-type bound as a corollary of our main theorem. 
\begin{thm}[Sharp $\frac{\pi}{2}$ Theorem]
\label{thm:pi-2}There exists a universal constant $C>0$ such that
for every Boolean function $f:\{\pm1\}^{n}\to\{\pm1\}$ with sufficiently
small $\theta=\max_{i}|\widehat{f}_{\{i\}}|$, 
\[
W_{1}(f)=\sum^{n}_{i=1}\widehat{f}^{2}_{\{i\}}\le\frac{2}{\pi}+C\theta^{2}.
\]
Moreover, the quadratic rate is attained by majority functions  in the sense of order, and
thus, is optimal in the order. 
\end{thm}

This theorem improves the existing $\frac{\pi}{2}$ theorem \cite{ODonnell14analysisof} where the
latter stating $W_{1}(f)\le\frac{2}{\pi}+C\theta.$ Our upper bound
is asymptotically optimal as $\theta\to0$. For odd $n$, consider the majority
function $g(x)=\operatorname{sgn}(\sum^{n}_{i=1}X_{i})\big)$. For
this function, using the refined Stirling expansion,  as $n\to\infty$,
\[
\theta=\big|\widehat{g}_{\{i\}}\big|=\frac{1}{2^{n-1}}{n-1 \choose \frac{n-1}{2}}=\sqrt{\frac{2}{\pi(n-1)}}\left(1-\frac{1}{4(n-1)}+O(n^{-2})\right),
\]
and 
\[
W_{1}(g)=\frac{n}{4^{n-1}}\binom{n-1}{\frac{n-1}{2}}^{2}=\frac{2}{\pi}\left(1+\frac{1}{2(n-1)}+O(n^{-2})\right).
\]
Therefore, as $\theta\to0$, 
\[
W_{1}(g)=\frac{2}{\pi}+\frac{\theta^{2}}{2}+O(\theta^{4}),
\]
confirming the optimality of the quadratic order.  

Denote $a=\mathbb{E}f$. By applying Chang's bound, the present author
\cite{yu2023phi} derived a bound on $W_{1}$ that is sharp for $a=\frac{1}{2},\theta=\frac{1}{4}$
and asymptotically sharp for $a=\frac{1}{2},\theta\uparrow\frac{1}{2}$.
By applying the variant of Khintchine's inequality of König, Schütt,
and Tomczak-Jaegermann in \eqref{eq:-3}, the present author \cite{yu2023phi}
derived another bound: $W_{1}(f)\le\frac{2}{\pi}+C\theta$ with explicit
$C$.  Further improvements of FKN-type bound  can be  also found in  \cite{yu2025average}.

\subsection{Proof of Theorem \ref{thm:pi-2} }

Associate to the linear Fourier spectrum of $f$ the weighted Rademacher
sum 
\[
Z=\sum^{n}_{i=1}\widehat{f}_{\{i\}}X_{i},
\]
where $X_{i}$ are independent Rademacher variables. By orthogonality
of characters, 
\[
\mathbb{E}[Z^{2}]=W_{1}(f),\qquad\mathbb{E}[fZ]=\sum_{i}\widehat{f}_{\{i\}}\mathbb{E}[fX_{i}]=\sum_{i}\widehat{f}^{2}_{\{i\}}=W_{1}(f).
\]
Since $f(x)\in\{\pm1\}$, we have the pointwise inequality $fZ\le|Z|$.
Taking expectation yields the key deterministic spectral bound: 
\begin{equation}
W_{1}(f)=\mathbb{E}[fZ]\le\mathbb{E}|Z|.\label{eq:-9}
\end{equation}

Let $s=\sqrt{W_{1}(f)}$ and define normalized weights 
\[
w_{i}=\frac{\widehat{f}_{\{i\}}}{s}.
\]
These weights satisfy the normalization $\norm{w}_{2}=1$ and maximal
coefficient bound 
\[
\norm{w}_{\infty}=\frac{\theta}{s}=:\beta.
\]
By our Theorem \ref{thm:main}, for all sufficiently small $\beta>0$,
\[
\mathbb{E}\left|\sum^{n}_{i=1}w_{i}X_{i}\right|=\sqrt{\frac{2}{\pi}}+O(\beta^{2}).
\]
Rescaling by $s$ recovers the expectation for the unnormalized sum
$Z$: 
\begin{equation}
\mathbb{E}|Z|=s\cdot\left(\sqrt{\frac{2}{\pi}}+O\left(\frac{\theta^{2}}{s^{2}}\right)\right)=\sqrt{\frac{2}{\pi}}\,s+O\left(\frac{\theta^{2}}{s}\right).\label{eq:-10}
\end{equation}
Combining \eqref{eq:-9} and \eqref{eq:-10}, we obtain 
\[
s^{2}\le\sqrt{\frac{2}{\pi}}\,s+O\left(\frac{\theta^{2}}{s}\right).
\]
Multiplying both sides by $s$ eliminates the denominator: 
\begin{equation}
s^{3}\le\sqrt{\frac{2}{\pi}}\,s^{2}+O(\theta^{2}).\label{eq:-11}
\end{equation}

Rewrite \eqref{eq:-11} as 
\[
s^{2}\left(s-\sqrt{\frac{2}{\pi}}\right)\le C\theta^{2},
\]
for some universal constant $C>0$. For small $\theta$, the right-hand
side is negligible, which forces $s$ to be close to $\sqrt{2/\pi}$.
Precisely, suppose $s=\sqrt{2/\pi}+\delta$ for $\delta>0$. Then
$\frac{2}{\pi}\,\delta\le C\theta^{2}$, i.e., $\delta=O(\theta^{2}).$
This yields the tight bound 
\[
s\le\sqrt{\frac{2}{\pi}}+O(\theta^{2}).
\]
Squaring both sides   gives
\begin{equation}
s^{2}=W_{1}(f)\le\frac{2}{\pi}+O(\theta^{2}).\label{eq:-12}
\end{equation}

\section*{Acknowledgments}

\emph{Generative-AI use disclosure:} During the preparation of this
manuscript, the author used ChatGPT (GPT-5.5-mini model) as an auxiliary
tool for exploring proof ideas and improving exposition. The authors take full responsibility for all mathematical claims.

\bibliographystyle{unsrt}
\bibliography{ref}

\end{document}